\documentclass[letterpaper,11pt]{article}

\usepackage{subfig}
\usepackage{graphicx}
\usepackage{amssymb}
\usepackage{amsthm}
\usepackage{amsmath}
\usepackage{epstopdf}
\usepackage{authblk}
\usepackage{cite}
\usepackage[margin=1.5in,includefoot]{geometry}
\usepackage{palatino}
\usepackage{microtype}

\usepackage{setspace}

\newtheorem{thm}{THEOREM}[section]
\newtheorem{lem}[thm]{LEMMA}
\newtheorem{cor}[thm]{COROLLARY}

\newtheorem{mydef}[thm]{DEFINITION}

\newtheorem{algorithm}{ALGORITHM}
\newtheorem{assumption}{ASSUMPTION}

\usepackage{hyperref}

\title{Fast integrators for dynamical systems with several temporal scales}
\author[1]{Yoonsang Lee\thanks{ylee@cims.nyu.edu}}
\author[2]{Bjorn Engquist\thanks{engquist@math.utexas.edu}}
\affil[1]{Courant Institute of Mathematical Sciences, New York University}
\affil[2]{Department of Mathematics and ICES, University of Texas at Austin}
\date{}

\begin{document}
\maketitle

\begin{abstract}
We propose a fast integrator to a class of dynamical systems with several temporal scales. The proposed method is developed as an extension of the variable step size Heterogeneous Multiscale Method (VSHMM) \cite{VSHMM}, which is a two-scale integrator developed by the authors. While iterated applications of multiscale integrators for two different scales increase the computational complexity exponentially as the number of different scales increases, the proposed method, on the other hand, has computational complexity linearly proportional to the number of different scales. This efficiency is achieved by solving different scale components of the vector fields with variable time steps. It is shown that variable time stepping of different force components has an effect of fast integration for the effective force of the slow dynamics. The proposed fast integrator is numerically tested on problems with several different scales which are dissipative and highly oscillatory including multiscale partial differential equations with sparsity in the solution space.
\end{abstract}

\section{Introduction}
Many problems in science and engineering have a wide range of temporal scales. From 
femto seconds of molecular dynamics to years of atmosphere, different range of scales interact with one another resulting in non-trivial dynamics on each scale. 
But direct modeling or numerical computation of the full scale is computationally prohibitive. For example, if the slow dynamics have a time scale of order $\mathcal{O}(1)$ while the fast dynamics have a time scale of order $\mathcal{O}(\frac{1}{\epsilon})$, it is required to use a time step of order $\mathcal{O}(\epsilon)$ for a simulation of up to a time $\mathcal{O}(1)$, which is computationally expensive for $\epsilon\ll1$.

In numerical analysis, this kind of modeling has a form of stiff problems which require an extremely small time step for numerical stability. For general stiff problems, there is a class of fast integrators such as exponential integrators or Gautschi type methods \cite{EXPINT,GAUTSCHI}. These methods use an analytic formula for the stiff part resulting in significantly less restriction on the time steps from stability and accuracy. However their applications are restricted to a class of problems in which the analytic solution operators of the significant parts are known a priori.

Instead of approximating all the scales, several multiscale methods have been developed and studied which approximate only slowly varying dynamics for computational efficiency. The equation-free method \cite{EQfree} solves the full fine scales for short times to capture the small scale behaviors and uses this information to approximate the essential coarse scale behaviors. The heterogeneous multiscale methods \cite{HMM} is a general framework for multiscale methods which incorporate small scale simulations on local domains in space and time to approximate the coarse scale solutions. One common characteristic of these methods is that they focus on approximating the slow variables by resolving the ergodic behaviors or measures of the fast variables using a fine time step locally, which yields computational savings of the methods. 
By imposing ergodicity on the fast dynamics conditional to the slow dynamics, averaging equations \cite{AVG1,AVG2} serve as an approximation of the slow dynamics. 
Thus, numerically resolving the ergodic measure of the fast variables is an important ingredient to develop robust and fast multiscale methods. For example, using a kernel with some special properties such as smooth regularity and particular decaying rates, a fast convergence of the invariant measure can be achieved \cite{HMMODE,BFHMM}. 

In this paper we focus on a more general class of problems than two-scale problems and propose a new multiscale integrator which captures the hidden slow variables of the following form of problems which have several temporal scales:
\begin{equation}\label{eq:VSSHMM:model2}
\frac{dx}{dt}=f_0(x)+\sum_{k=1}^K\frac{f_k(x)}{\epsilon_k},\quad x(0)=x_0,\quad 0<\epsilon_K, \epsilon_{k+1}\leq\epsilon_{k}, \epsilon_1\ll 1.
\end{equation}
The slow variables and fast variables are defined as follows
\begin{mydef}\label{def:slowvariable}
a function $\xi:x\in D\to\mathbb{R}$ is a slow variable of the system \eqref{eq:VSSHMM:model2} if it has a bounded derivative for $0<\epsilon_1<\epsilon_0$ for a positive constant $\epsilon_0\ll 1$ along the flow $x(t)$ in $D$, that is,
\begin{equation}
\sup_{x\in D,\epsilon_1\in (0,\epsilon_0)}|\nabla \xi\cdot \frac{dx}{dt}|\leq C
\end{equation} 
where $D\subset\mathbb{R}^n$ is an open connected set  and $C$ is a constant independent of $\epsilon$. If $C$ is dependent on $\epsilon_1$ and unbounded as $\epsilon_1\to 0$, $\xi$ is a fast variable of the system.
\end{mydef}
There are methods which require the identification of the hidden slow variables in advance \cite{AKST,ALT} but our method does not require the identification of the slow variables yet automatically approximates them; see \cite{Resonance,MSHMM,FLAVORS,BFHMM} for two-scale integrators which do not require a priori identification of the slow variables to approximate them.

Two-scale integrators can be applied to the several time scale problems through iterated applications of the methods \cite{3scHMM}. That is, a two-scale integrator is employed to resolve the slow variables while the ergodic measure of the intermediate scale variables is resolved by another application of the two-scale method (figure \ref{fig:VSSHMM:iterated} shows a schematic of iterated application of two-scale HMM integrator).  Under the existence of a deterministic averaging equation (which is of our interest in this paper), this iterated approach shows satisfactory slow variable approximation skills. But the computational cost of this approach increases exponentially as the number of separated scales increases. 
The method we propose here, on the other hand, solves each intermediate and stiff vector filed components for the same number of times with different time steps according to their scales and thus the new method has computational complexity increasing linearly depending on the number of different scales (see figure \ref{fig:VSSHMM:vsshmm} for a schematic of the proposed method).
\begin{figure}
\centering
\subfloat[Iterated application of two-scale integrator HMM\label{fig:VSSHMM:iterated}]{\includegraphics[width=.45\textwidth]{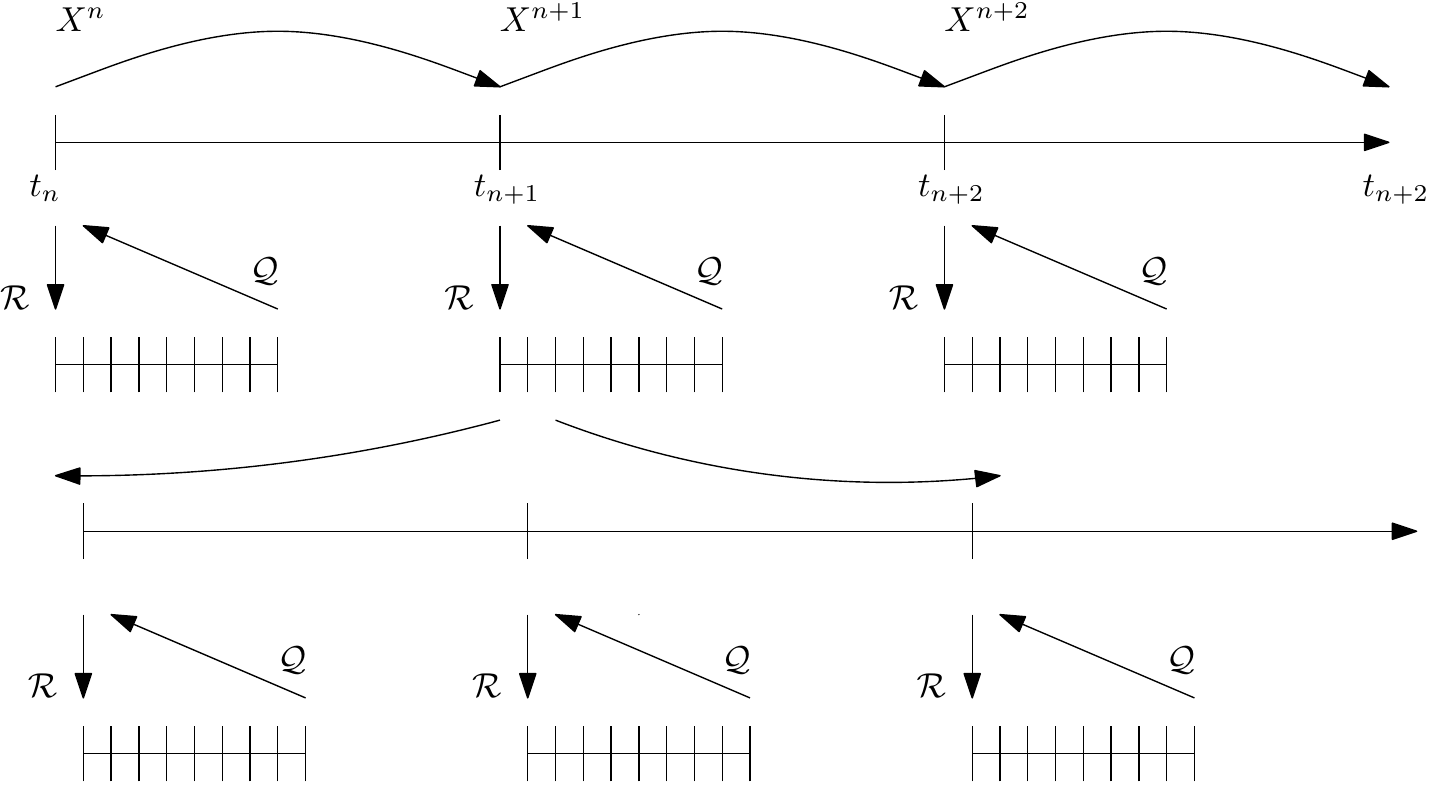}}
\subfloat[VSHMM\label{fig:VSSHMM:vsshmm}]{\includegraphics[width=.45\textwidth]{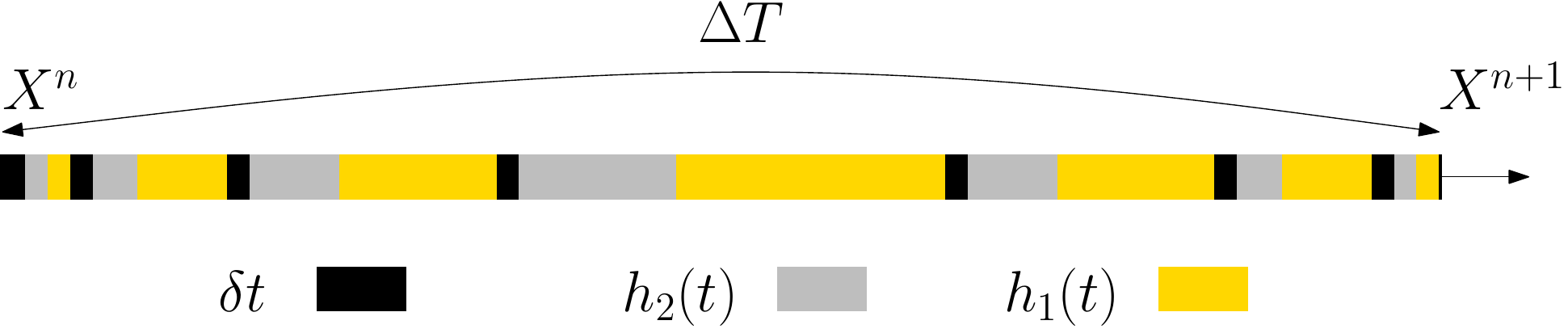}}
\caption{Schematics of a iterated-HMM and VSSHMM}\label{fig:VSSHMM:sche}
\end{figure}

The time step for each scale is not necessarily constant. In \cite{VSHMM}, a variable step size Heterogeneous Multiscale Method (VSHMM) was shown to have higher accuracy than constant time stepping for two-scale problems. We employ the same strategy here and extend the VSHMM to improve the accuracy of the proposed integrator for several time scales. Variable time stepping has another effect in addition to improved accuracy; variable time stepping for intermediate and fast scale variables can be seen as a Monto Carlo integration in a high dimensional space although the method does not include any random sampling for integration. 

Although the proposed method has the aforementioned benefits and is applicable to systems with several temporal scales, we have to mention that the existence of deterministic averaging equation of slow variables is essential to develop our method. 
For general several time scale problems, the longest time scale limit can be stochastic \cite{SDElimit}. 
Thus in this study we focus on systems which posses deterministic limit equations (an extension of our method to handle stochastic limits of several temporal scale problems will be reported in the near future). We also note the reader that there are efficient methods for stiff dissipative systems without scale separation; implicit methods for small dimensional systems and Chebyshev method for large dimensional systems. We include stiff dissipative systems with potential application of the proposed method to concurrent multiscale problems; see \cite{HMM} for concurrent multiscale problems using HMM and references therein.

This paper is organized in the following way. In section \ref{sec:averaging}, we discuss averaging equations for several temporal scale problems with deterministic limits. For dissipative stiff problems, we show that  the different time scales can be treated as one time scale separated from the time scale of slow variables. For highly oscillatory problems, we show that the calculation of the effective force is related to an integration in a high dimensional space. The fast integrators for several temporal scale problems are then proposed in section \ref{sec:VSHMM}. Section \ref{sec:numerical} presents numerical examples with fast relaxation and a highly oscillatory behavior to validate the proposed method with concluding remarks in section \ref{sec:conclusion}.

\section{Averaging equations}\label{sec:averaging}
In this section, we discuss averaging equations of several temporal scale problems with deterministic limits, which play a crucial role in the development of the fast integrator. The time integrator proposed in section \ref{sec:VSHMM} achieves high efficiency by approximating the averaging equations of the slow dynamics. Without loss of generality, for an easy exposition of the key idea, we consider three different scale problems 
\begin{equation}
\frac{dx}{dt}=f_0(x)+\frac{f_1(x)}{\epsilon_1}+\frac{f_2(x)}{\epsilon_2},\quad x(0)=x_0,\quad 0<\epsilon_2<\epsilon_1\ll 1.
\end{equation}
For the existence of hidden slow variables, we assume that 
\begin{assumption}\label{assumption:diffeomorphism}There exists a $C^1$ diffeomorphism $\Psi:x\in\mathbb{R}^{n}\to(\xi,\eta,\zeta)\in\mathbb{R}^{n_1+n_2+n_3}$ where $(\xi,\eta,\zeta)$ satisfies
\begin{equation}\label{eq:3scmodel}
	\begin{split}
	\frac{d\xi}{dt}&=f(\xi,\eta,\zeta)\\
	\frac{d\eta}{dt}&=\frac{1}{\epsilon_1}g(\xi,\eta,\zeta)\\
	\frac{d\zeta}{dt}&=\frac{1}{\epsilon_2}h(\xi,\eta,\zeta)\\
	(\xi(0),\eta(0),\zeta(0))&=(\xi_0,\eta_0,\zeta_0)
	\end{split}
\end{equation}
for $\xi\in\mathbb{R}^{n_1}$, $\eta\in\mathbb{R}^{n_2}$ and $\zeta\in\mathbb{R}^{n_3}$, $n=n_1+n_2+n_3$. 
\end{assumption}
Based on Definition \ref{def:slowvariable}, $\xi$ is the slow variable and $\eta$ and $\zeta$ are the fast variables of the system which have unbounded derivative for $\epsilon\to0$. For the existence of an averaging equation, we further assume the following
\begin{assumption}\label{assumption:general}For fixed $\xi$, the other variables $\eta$ and $\zeta$ are ergodic with an invariant measure $\mu^{\eta,\zeta}_{\xi}(\eta,\zeta)$.
\end{assumption}
Under this assumption, there exists a deterministic averaging equation 
\begin{equation}\label{eq:effective}
	\begin{split}
	\frac{d\Xi}{dt}&=F(\Xi)\\
	\Xi(0)&=\xi_0
	\end{split}
\end{equation}
where
$F(\Xi)=\iint f(\Xi,\eta,\zeta)d\mu^{\eta,\zeta}_{\xi}(\eta,\zeta)$.
The solution of the averaged equation $\Xi$, which approximate the slow variable $\xi$ of the original equation, is normally expected to have the following type of error bound (see \cite{AVG1,HMMODE,HMM} for example)
\begin{equation}\label{eq:avgerror}
\|\xi-\Xi\|_{\infty}\leq C\epsilon_1^a
\end{equation}
for constants $a>0$ and $C$ which are independent of $\epsilon$ over the time interval of interest. Under this assumption we can apply a two-scale integrator such as HMM \cite{HMMODE} to numerically approximate the invariant measure of $(\eta,\zeta)$ but this computation is also another two-scale problem which is computationally expensive for $\frac{\epsilon_2}{\epsilon_1}\ll1$. In \cite{3scHMM,iteratedHMM}, iterated application of the two-scale HMM was successfully applied to overcome the aforementioned problem under an ergodicity assumption of $\zeta$ along for fixed $\eta$ in addition to $\xi$.

Our goal of the fast integrator proposed in this study is a development of a method whose computational complexity is significantly less than the iterated application of two-scale integrators. Note that the later approach has an exponentially increasing computational cost as the number of different scales increases. On the other hand, our proposed will be shown in section \ref{sec:VSHMM} has a linearly increasing computational cost as the number of different scales increases. Before we describe the fast integrator in the next section, we first study analytic results, averaging equations of several scale problems, which play an important role in the development of the fast integrator. The main point of the following result is that the $\epsilon_2$ scale can be relaxed to the $\epsilon_1$ scale and this relaxation provides the same averaging equation of the original equation without relaxation. We first consider a dissipative system whose Jacobian has eigenvalues with only negative real parts followed by highly oscillatory systems.

\subsection{Dissipative case}
We assume the following fact for dissipative systems of the form of \eqref{eq:3scmodel} which guarantees the existence of iterated averaging equation for $\xi$.

\begin{assumption}\label{assumption:dissipative}\end{assumption}
\begin{itemize}
\item For fixed $\xi$ and $\eta$, the Jacobian of $h$, $\partial_{\zeta}h$, has eigenvalues of negative real parts and $\zeta$ converges to a unique fixed point $\zeta^*:(\xi,\eta)\in\mathbb{R}^{n_1+n_2}\to\mathbb{R}^{n_3}$ such that $h(\xi,\eta,\zeta^*(\xi,\eta))=0$. That is, $\zeta$ has a unique invariant measure $\delta(\zeta-\zeta^*(\xi,\eta))$ for fixed $\xi$ and $\eta$. We further assume that $\zeta^*$ is differentiable in $\eta$.
\item For fixed $\xi$, the Jacobian of $\tilde{g}:=g(\xi,\eta,\zeta^*(\xi,\eta))$ has eigenvalues of negative real parts and $\eta$ converges to a unique fixed point $\eta^*:\xi\subset\mathbb{R}^{n_1}\to\mathbb{R}^{n_2}$ such that $\tilde{g}(\xi,\eta^*(\xi))=0$. That is, for  $\epsilon_2\to 0$, $\eta$ is ergodic with a unique invariant measure $\delta(\eta-\eta^*(\xi))$ for fixed $\xi$.
\end{itemize}

For a dissipative system of the form of \eqref{eq:3scmodel} satisfying the above assumption for iterated averaging, we show that the iterated averaging equation is the same as the averaging equation of the following two-scale problems obtained by relaxing $\epsilon_2$ to $\epsilon_1$
\begin{equation}\label{eq:2scmodel}
	\begin{split}
	\frac{d\xi}{dt}&=f(\xi,\eta,\zeta),\\
	\frac{d\eta}{dt}&=\frac{1}{\epsilon_1}g(\xi,\eta,\zeta),\\
	\frac{d\zeta}{dt}&=\frac{1}{\epsilon_1}h(\xi,\eta,\zeta).
	\end{split}
\end{equation}
Before we prove theorem \ref{thm:dissipative}, we need a technical lemma. The following arguments are for fixed $\xi$ and thus we suppress $\xi$ for simplicity of exposition.
\begin{lem}\label{lem:diffofinverse}
$\displaystyle\left.\frac{\partial \zeta^*(\eta)}{\partial \eta}\right|_{\eta=\eta^*}=-(\partial_{\zeta} h)^{-1}\partial_{\eta} h$
\end{lem}
\begin{proof} 
For $\delta\in\mathbb{R}^{n_2}$ such that $0<\|\delta\|\ll 1$, let $h(\eta^*,\zeta^*_0)=0$ and $h(\eta^*+\delta,\zeta^{*}_{\delta})=0$ where $\zeta^*_0=\zeta^*(\eta^*)$ and $\zeta^*_{\delta}=\zeta^*(\eta^*+\delta)$. From Taylor series expansion at $(\eta^*,\zeta^*$), we have
\begin{equation}
\begin{split}
0=h(\eta^*+\delta,\zeta^{*}_{\delta})=&\partial_{\eta}h(\eta^*,\zeta^*_0)\delta+\partial_{\zeta}h(\eta^*,\zeta^*_0)(\zeta^{*\delta}-\zeta^*_0)+\mathcal{O}(\|\delta\|^2)\\
&+\mathcal{O}(\|\zeta^{*}_{\delta}-\zeta^*_0\|^2)+\mathcal{O}(\|\delta\| \|\zeta^{*}_{\delta}-\zeta^*_0\|)
\end{split}
\end{equation}
Invertibility of $\partial_{\zeta}h$ (which is guaranteed by eigenvalues with negative real parts) and differentiability of $\zeta^*(\eta)$ in $\eta$ implies that $\mathcal{O}(\|\zeta^{*}_{\delta}-\zeta^*_0\|)=\mathcal{O}(\|\delta\|)$. Thus we have
\begin{equation}
\zeta^{*}_{\delta}-\zeta^*_0=-(\partial_{\zeta}h)^{-1}\partial_{\eta}h\delta+(\partial_{\zeta}h)^{-1}\mathcal{O}(\|\delta\|^2)
\end{equation}
which proves the lemma.
\end{proof}

We now prove the main theorem for dissipative systems with an iterated averaging equation.
\begin{thm}\label{thm:dissipative}
For fixed $\xi$, under Assumption \ref{assumption:dissipative}, the unique fixed point of the three scale problem \eqref{eq:3scmodel}, $(\eta^*,\zeta^*_0)=(\eta^*(\xi),\zeta^*(\xi,\eta^*(\xi)))$, is also a unique fixed point of the two-scale problem \eqref{eq:2scmodel} and is asymptotically stable.
\end{thm}
\begin{proof}
It is straightforward that $(\eta^*,\zeta^*_0)$ is a fixed point of \eqref{eq:2scmodel}. Let $(\tilde{\eta},\tilde{\zeta})$ is a fixed point of \eqref{eq:2scmodel}. From the uniqueness of $\zeta^*$, for fixed $(\xi,\tilde{\eta})$, we have $\tilde{\zeta}=\zeta^*(\tilde{\eta})$ and from the uniqueness of $\tilde{g}(\xi,\eta)=g(\xi,\eta,\zeta^*(\xi,\eta))$ for fixed $\xi$, $\tilde{\eta}=\eta^*$. Thus $(\eta^*,\zeta^*_0)$ is a unique fixed point of \eqref{eq:2scmodel}.

Next we show asymptotic stability of $(\eta^*,\zeta^*_0)$. For this, it suffices to show that the Jacobian matrix of \eqref{eq:2scmodel}
\begin{equation}\label{eq:jacobian}
	\left.\begin{pmatrix}\partial_{\eta} g& \partial_{\zeta} g\\ \partial_{\eta} h& \partial_{\zeta} h\end{pmatrix}
\right|_{\eta^*,\zeta^*_0}
\end{equation}
has eigenvalues of negative real parts. By eliminating $\partial_{\zeta} g$, we have
\begin{equation}
	\begin{pmatrix}\partial_{\eta} g-\partial_{\zeta} g(\partial_{\zeta} h)^{-1}\partial_{\eta} h& 0\\ \partial_{\eta} h& \partial_{\zeta} h\end{pmatrix}
\end{equation}
From Assumption \ref{assumption:dissipative}, $\partial_{\eta} \tilde{g}=\partial_{\eta}g+\partial_{\zeta}g\frac{\partial\zeta^*}{\partial \eta}$ has eigenvalues of negative real part. Thus, using Lemma \ref{lem:diffofinverse}, \eqref{eq:jacobian} has eigenvalues of negative real parts which implies asymptotic stability using a Lyapunov function $V=\frac{1}{2}(\|\eta\|^2+\|\zeta\|^2)$ \cite{Arnold,HS1974}.
\end{proof}

Theorem \ref{thm:dissipative} implies that the scale separation between $\eta$ and $\zeta$ is not necessary to resolve the invariant measure of  $(\eta,\zeta)$ in the iterative averaging (which is a Dirac measure for dissipative systems under Assumption \ref{assumption:dissipative}). By relaxing the fast scale $\mathcal{O}(1/\epsilon^2)$ to slower scale $\mathcal{O}(1/\epsilon)$ (but still fast enough compared to the slow scale $\mathcal{O}(1)$ variable  of our interest $\xi$) we can obtain the same invariant measure of the iterated averaging equation. Thus the relaxed two-scale problem \eqref{eq:2scmodel} has the same averaging equation as the three scale problem \eqref{eq:3scmodel}.

\begin{cor}
For fixed $\xi$, the three scale problem \eqref{eq:3scmodel} and the relaxed two-scale problem \eqref{eq:2scmodel} have the same averaging equation.
\end{cor}

The uniqueness of the fixed point in Assumption \ref{assumption:dissipative} is strong and seldom satisfied in general. But we have uniqueness locally; for a small open neighborhood of a fixed point, negative real parts of all eigenvalues imply that the Jacobian of the fast dynamics is invertible and by the implicit function theorem, the fixed point is unique. Thus, if the initial conditions $\eta_0$ and $\zeta_0$ of $\eta$ and $\zeta$ respectively are sufficiently close to the fixed points $\eta^*$ and $\zeta^*$ for the initial $\xi_0$, the uniqueness is satisfied and we can relax $\mathcal{O}(\frac{1}{\epsilon_2})$ scale to $\mathcal{O}(\frac{1}{\epsilon_1})$ to approximate the invariant measure efficiently. The transient behavior of the initial value of the fast variables to the corresponding fixed points has another effect on the accuracy of the approximation by an averaging equation. In \cite{VSHMM}, for two-scale problems, it is emphasized that by successfully capturing the fast transient behavior of the fast variables, the approximation of the slow variable by the averaging solution increases from $a=1/2$ to $a=1$ in \eqref{eq:avgerror}. In section \ref{subsec:nume:dissipative} we check the importance of the resolved fast trasient behavior of the fast variables through a numerical example.

\subsection{Highly oscillatory case}
For general highly oscillatory problems with several scales where the invariant measure of the fast dynamics $\mu_{\xi}^{\eta,\zeta}(\eta,\zeta)$ is not a Dirac type, we cannot derive a theorem equivalent to \ref{thm:dissipative} for the dissipative case; the invariant measure of $\zeta$ is different for different $\eta$. Thus we relax the system \eqref{eq:3scmodel} by assuming that the fast variables have no direct interactions; that is the dynamics of $\eta$ and $\zeta$ for a fixed $\xi$ are independent of each other :
\begin{equation}\label{eq:simpler3scmodel}
	\begin{split}
	\frac{d\xi}{dt}&=f(\xi,\eta,\zeta),\\
	\frac{d\eta}{dt}&=\frac{1}{\epsilon_1}g(\xi,\eta),\\
	\frac{d\zeta}{dt}&=\frac{1}{\epsilon_2}h(\xi,\zeta).
	\end{split}
\end{equation}
Although there are no direct interactions between the fast variables $\eta$ and $\zeta$ in \eqref{eq:simpler3scmodel}, the fast variables have indirect interactions through the slow variable as each fast variable has an interaction with the slow variable.
Under this setting, the invariant measure $\mu_{\xi}^{\eta,\zeta}(\eta,\zeta)$ of Assumption \ref{assumption:general} can be represented by a product of two measures $\mu_{\xi}^{\eta}(\eta)$ and $\mu_{\xi}^{\zeta}(\zeta)$ which are the invariant measures of $\eta$ and $\zeta$ respectively. Thus the averaging equation \eqref{eq:effective} has the following form
\begin{equation}\label{eq:VSSHMM:effective}
\frac{d\Xi}{dt}=F(\Xi):=\iint f(\Xi,\eta,\zeta)d\mu_{\Xi}^{\eta}(\eta)d\mu_{\Xi}^{\zeta}(\zeta).
\end{equation}

If we can identify the slow and fast variables, there is a natural method whose complexity increases linearly as the number of fast variables increases. The idea is to solve the two fast variables independently for sufficiently long time (but short compared to the time step of the slow variable for computational efficiency). From these local calculations, say there are $N$ and $M$ steps of evolution of $\eta$ and $\zeta$ respectively, we find all possible combinations between $\eta_1$ and $\eta_2$ to compute the averaging force 
\begin{equation}\label{eq:uniform_int}
F=\sum_{n,m}^{N,M} f(\Xi,\eta_n,\zeta_m).\end{equation}
One disadvantage of this approach is that if there are many fast variables following different invariant measures then \eqref{eq:uniform_int} corresponds to integration in high dimensional space and thus the above approach is computationally expensive compared to Monte-Carlo integration. Also, \eqref{eq:uniform_int} requires a priori identification of the slow and fast variables (and their corresponding forcing formula) for the computation of the effective forcing of the averaging solution. 

In the next section, we propose our method which does not require a priori identification of the slow and fast variables yet the proposed method captures the hidden slow variables; see \cite{MSHMM, FLAVORS, BFHMM} for two-scale problems which do not require a priori identification of the slow and fast variables. Also one key ingredient of our method is to use variable stepping for intermediate time components which shows a Monte-Carlo effect in the integration of the effective forcing $F(\Xi)$.

\section{Variable Step Size Heterogeneous Multiscale Methods}\label{sec:VSHMM}
The key idea of VSHMM is to include different components of the force depending on the variable step size, from the full $f_{\epsilon}(x)$ for the shortest step size to only the slowest components for the longest step size. The intermediate step size will contain the intermediate to slow components of $f_{\epsilon}(x)$. That is, the method uses different time steps in the evolution of the system using the following split three problems,
\begin{align}
\frac{dx}{dt}=&f_0(x)+\frac{f_1(x)}{\epsilon_1}+\frac{f_2(x)}{\epsilon_2}\label{eq:VSSHMM:per3},\\
\frac{dx}{dt}=&f_0(x)+\frac{f_1(x)}{\epsilon_1}\label{eq:VSSHMM:per2},\\
\frac{dx}{dt}=&f_0(x)\label{eq:VSSHMM:per1}.
\end{align}
That is, for the evolution of each equation, we use time steps $\delta t$, $h_2$ and $h_1$ for (\ref{eq:VSSHMM:per3}), (\ref{eq:VSSHMM:per2}) and (\ref{eq:VSSHMM:per1}) respectively. The smallest time step $\delta t$ is constant while $h_1(t)$ and $h_2(t)$ are variable in time which are determined by special functions. We solve different scale parts for the same number of times as the slowest part and the elapsed time is determined by the sum of each time step, $\delta t+h_2+h_1$; thus it is straightforward that the computational complexity of our method increases linearly proportional to the number of different scales.

For two-scale problems, it is shown in \cite{VSHMM} that variable time steps can be employed to achieve high accuracy in the approximation of the hidden slow variables. Thus we use the variable time stepping for $h_1$ to achieve high accuracy. Here we further employ the variable time stepping for the intermediate scales ($h_2$ in our setting) and we show that variable mesoscopic time steps have an effect of Monte-Carlo integration in the estimation of the effective force of the slow variable. Note that for the dissipative case, if we can resolve the fast transient behavior of the fast variables using variable time steps (in this case, fine time steps at the beginning), then uniform time steps for each different scales are sufficient to approximate the invariant measure of the fast variables which is essential to compute the effective force.

For the determination of the mesoscopic time steps $h_1(t)$ and $h_2(t)$, we use a special function satisfying moment and regularity conditions. For a given $q\in\mathbb{N}$, let $K\in C^q_c((0,1))$ have a compact support in $(0,1)$ such that
\begin{eqnarray}
	\int_0^1K(t)dt&=&1,\\
	\frac{d^rK(t)}{dt^r}&=&0,\quad r=0,1,...,q \textrm{ for }t=0,1.
\end{eqnarray}
Contrary to general time integrators which provide solutions at each time step, VSHMM provides solutions only at the sampling time interval $\Delta T$ longer than the largest time step $h_1$; VSHMM has intermediate solution states other than the sampling interval but high accuracy of the solutions are guaranteed only at the sampling interval (see figure \ref{fig:VSSHMM:vsshmm}). For a given macro sampling time step $\Delta T$, savings factors $\alpha_1>\alpha_2>1$ and $m\in\mathbb{N}$, the two mesoscopic time steps, $h_1(t)$ and $h_2(t)$, for the evolution of (\ref{eq:VSSHMM:per1}) and (\ref{eq:VSSHMM:per2}) respectively, are given by 
\begin{align}
h_1(t)&=\alpha_1 {\delta t}{K}_{\Delta T,q}(\Theta_{\Delta T,q}^{-1}(t\mod \Delta T)),\\
h_2(t)&=\alpha_2 {\delta t}{K}_{\Delta T/m,q}(\Theta_{\Delta T,q}^{-1}(t\mod \frac{\Delta T}{m}))
\end{align}
when $K_{\Delta T,q}$ is a rescaled version of $K$,
$$K_{\Delta T,q}=\frac{1}{\Delta T}K(\frac{t}{\Delta T})$$
and $\Theta_{\Delta T,q}(t)$ is the antiderivative of ${K}_{\Delta T,q}$ with $\Theta_{\Delta T,q}(0)=0$.
It can be easily verified that $K_{\Delta T/m,q}$ satisfies the following moment and regularity conditions
\begin{eqnarray}
	\label{eq:kerprop1}\int {K}_{\Delta T/m,q}dt&=&\frac{\Delta T}{m},\\
	\label{eq:kerprop2}\frac{d^r {K}_{\Delta T,q}(t)}{dt^r}&=&0,\quad r=0,1,...,q,\quad t=0,\Delta T.
\end{eqnarray}

As mentioned above, the largest time step $h_1(t)$ is variable to achieve high accuracy and the intermediate mesoscopic time step $h_2(t)$ is also variable to capture the fast transient behavior of the dissipative problems or to expedite the convergence to the corresponding invariant measure for highly oscillatory problems. To see the later effect, consider the case when $\eta$ and $\zeta$ are periodic with periods $\epsilon_1$ and $\epsilon_2$ respectively. $\zeta$ is always evolved using a constant time step $\delta t$. If we use a constant time step $h_2$ for $\eta$, the evolution of $\eta$ and $\zeta$ will be a straight line in the $\eta-\zeta$ plane (see figure \ref{fig:VSSHMM:timeplaneconst}). If the ratio between $h_2$ and $\delta t$, say $\beta$, is irrational the straight line will cover a dense region of the torus after sufficiently long time. But it usually requires very long run and thus after only a few periodic motions, the points will still look like a straight line (see figure \ref{fig:VSSHMM:timeplaneconstmod} for the projected points onto the torus).
\begin{figure}
\centering
\subfloat[Constant time step\label{fig:VSSHMM:timeplaneconst}]{\includegraphics[width=0.45\textwidth]{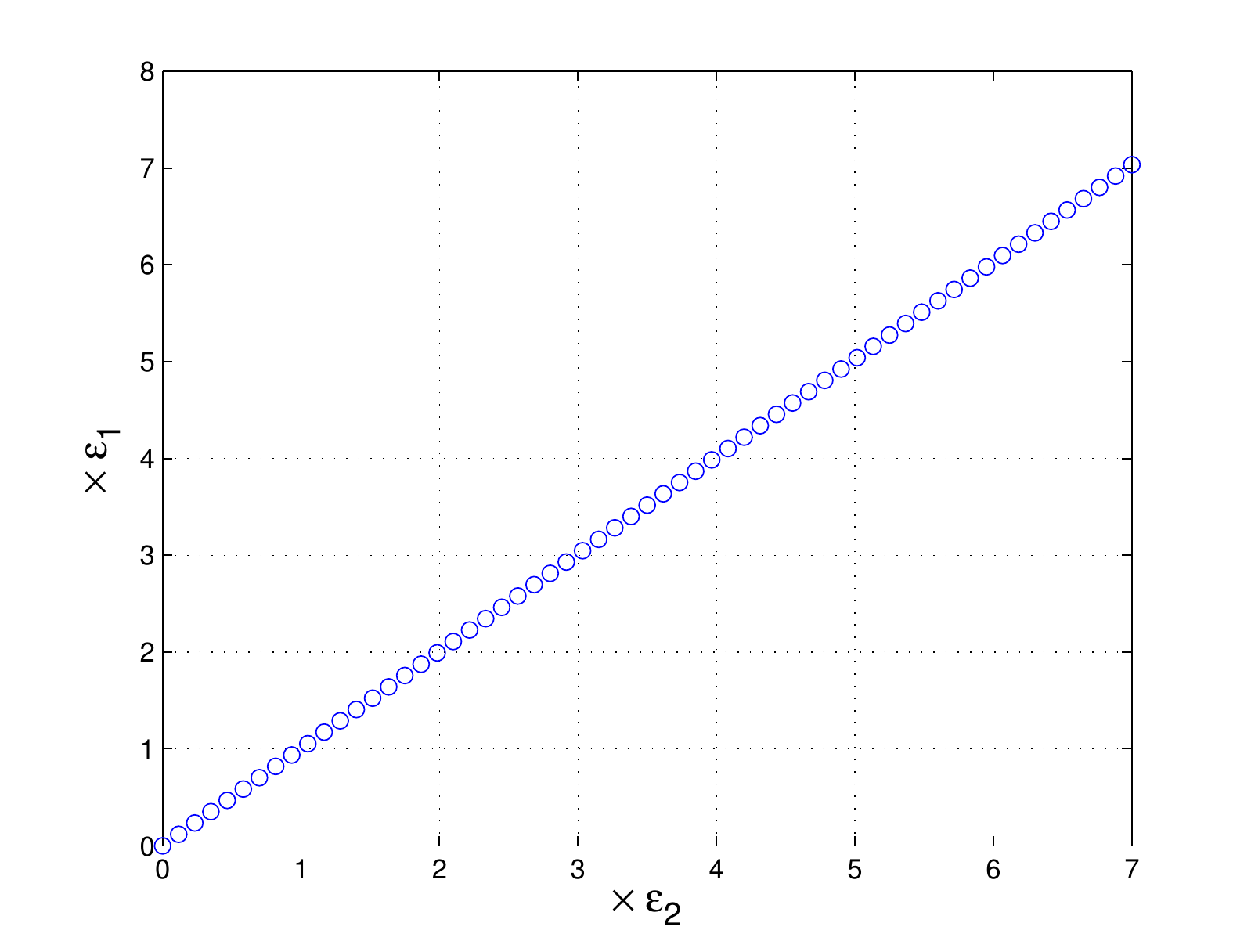}}
\subfloat[Projection onto torus\label{fig:VSSHMM:timeplaneconstmod}]{\includegraphics[width=0.45\textwidth]{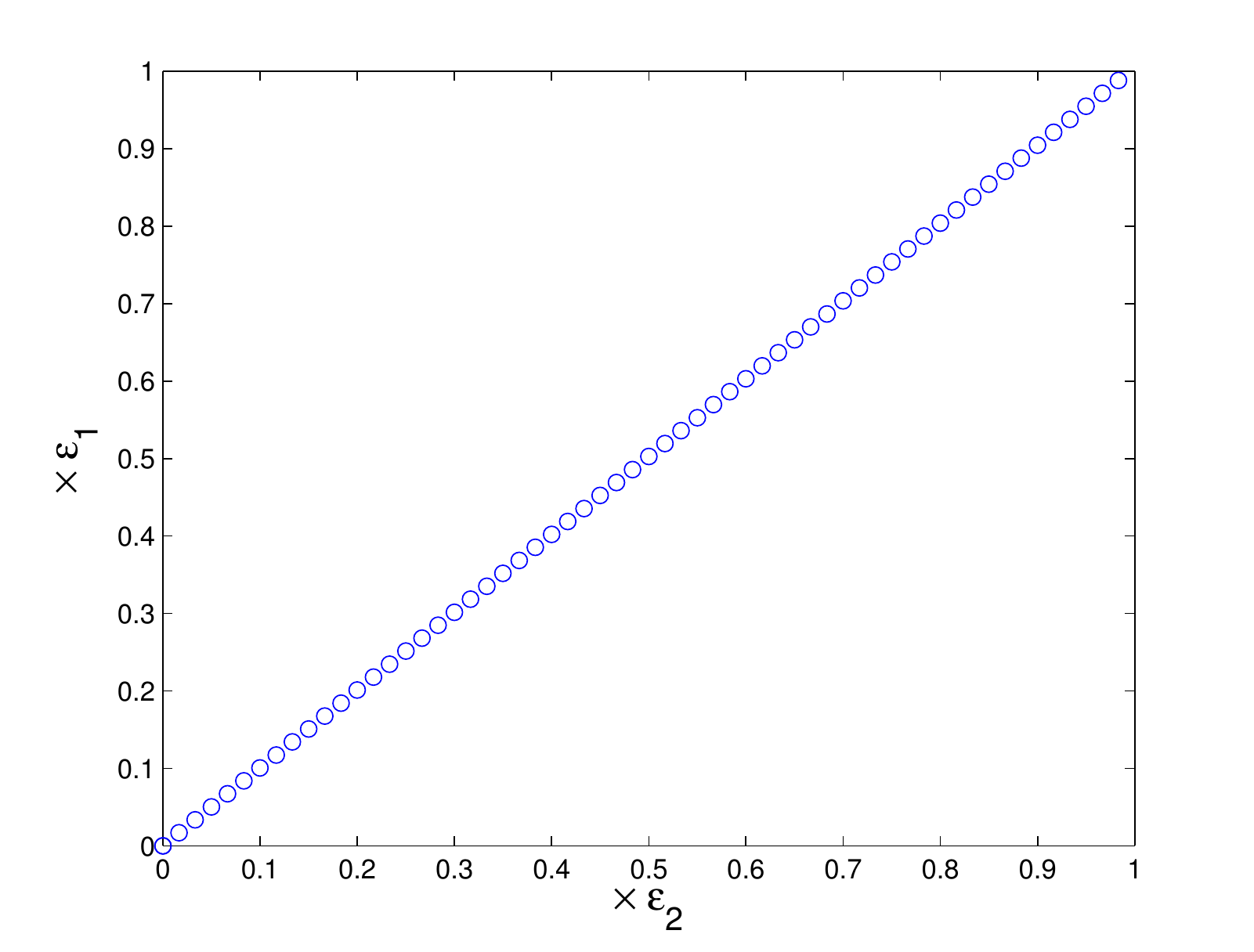}}
\caption{60 sampling points of constant time steps in several copies of a torus. $\beta=\sqrt{1.01}$}\label{fig:VSSHMM:time plane}
\end{figure}

On the other hand, the variable time stepping of $h_2$ becomes a curve (see figure \ref{fig:VSSHMM:timeplanevar} for the case when $K(t)=\left(1+\cos({2\pi (t-1/2)})\right)$) and its projection onto the torus (figure \ref{fig:VSSHMM:timeplanevarmod}) shows much wider spread for the sampling of the integration of the effective force \eqref{eq:effective}.
\begin{figure}
\centering
\subfloat[Variable time step \label{fig:VSSHMM:timeplanevar}]{\includegraphics[width=0.45\textwidth]{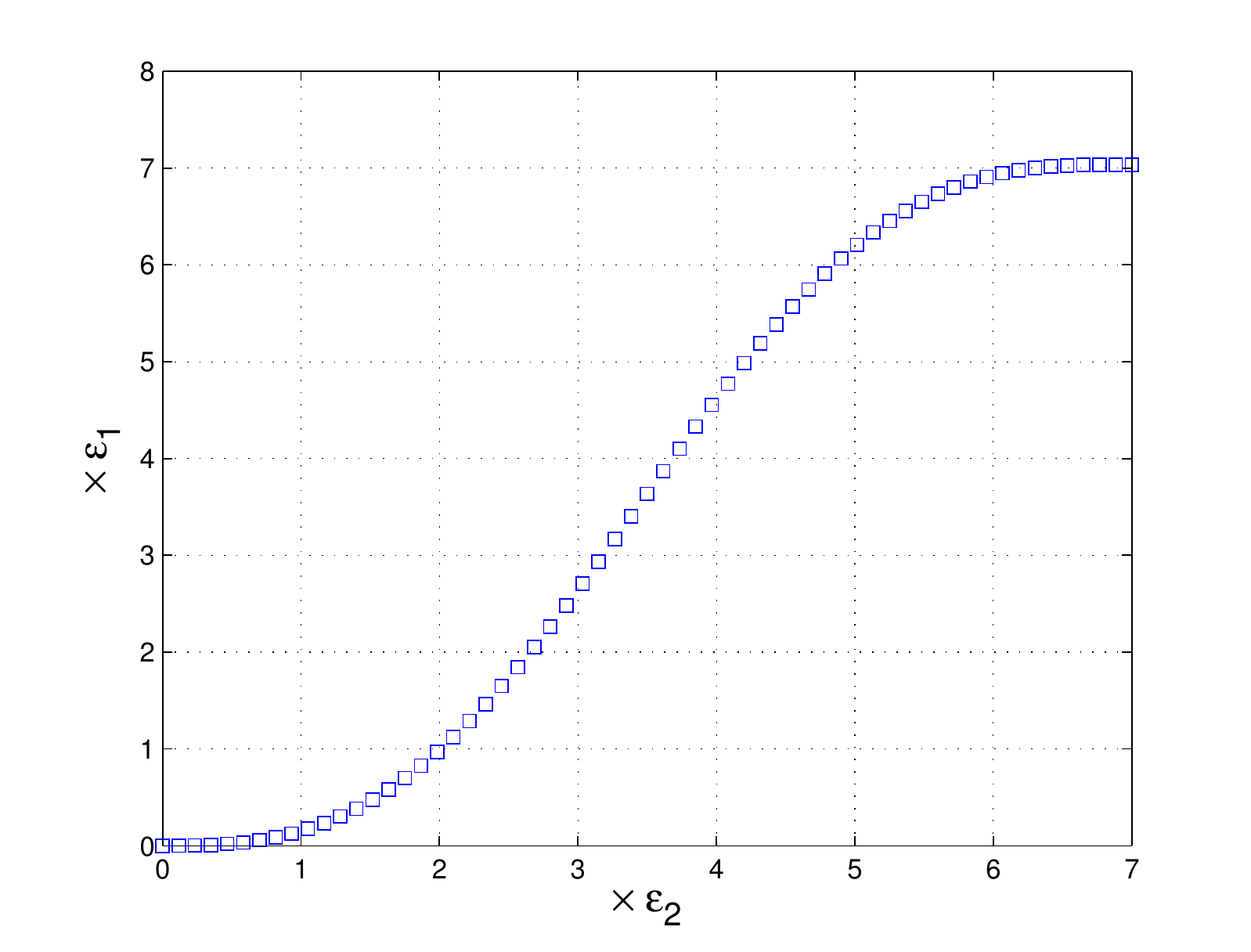}}
\subfloat[Projection onto torus\label{fig:VSSHMM:timeplanevarmod}]{\includegraphics[width=0.45\textwidth]{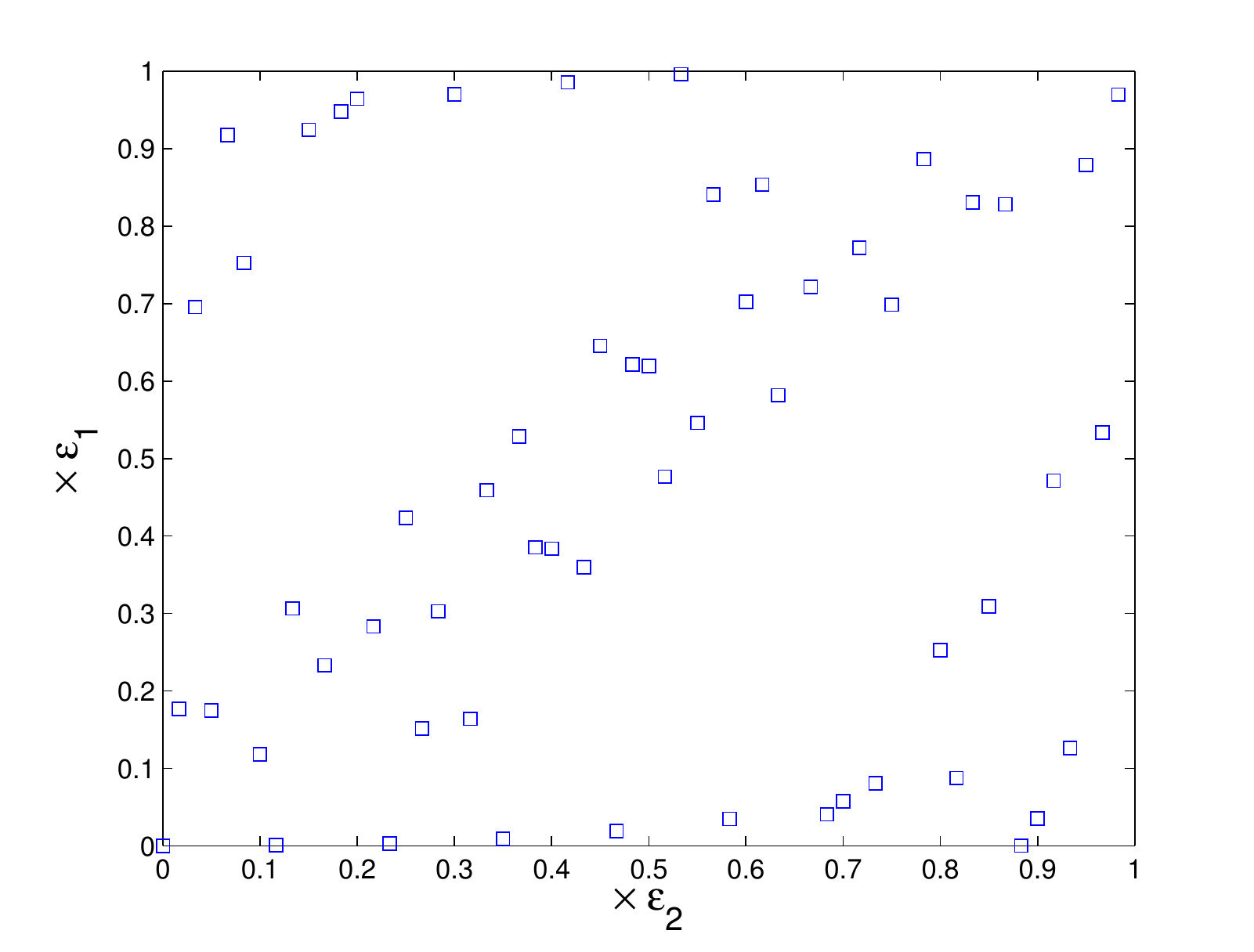}}
\caption{60 sampling points of variable time stpes in several copies of a torus. $\beta=\sqrt{1.01}$}\label{fig:VSSHMM:timeplaneprojected}
\end{figure}

The complete description of the algorithm is provided below.
\begin{algorithm}[one macro time step integration of VSHMM]
\end{algorithm}
Let $\tilde{x}^n$ be the VSSHMM solution to (\ref{eq:VSSHMM:model2}) at $t=t^n:=n\Delta T$ with savings factors $\alpha_1$ and $\alpha_2$.
\begin{enumerate}
\item Integrate the full system (\ref{eq:VSSHMM:per3}) for $\delta t$ to resolve the $\epsilon_2$ scale
$$\hat{x}(t+\delta t)=\Phi^{\epsilon_2}_{\delta t}\tilde{x}(t^n)$$
where $\Phi^{\epsilon_2}_{\delta t}$ is an integrator of (\ref{eq:VSSHMM:per3}) for $\delta t$.
\item Update time
$$t=t^n+\delta t.$$
\item Integrate the system without the $\epsilon_2$ scale term, (\ref{eq:VSSHMM:per2}), with the mesoscopic time step $h_2(t)$ 
$$\hat{x}(t+h(t))=\Phi^{\epsilon_1}_{h_2(t)}\hat{x}(t)$$
where $\Phi^{\epsilon_1}_{h_2(t)}$ in an integrator of (\ref{eq:VSSHMM:per2}) for $h_2(t)$.
\item Update time
$$t=t+h_2(t)$$
\item Integrate the system without the $\epsilon_1$ and $\epsilon_2$ scale terms, (\ref{eq:VSSHMM:per1}), with the mesoscopic time step $h_1(t)$ 
$$\hat{x}(t+h(t))=\Phi^{\epsilon_1}_{h_1(t)}\hat{x}(t)$$
where $\Phi^{0}_{h_1(t)}$ in an integrator of (\ref{eq:VSSHMM:per1}) for $h_1(t)$.
\item Update time
$$t=t+h_1(t)$$
\item If time approaches the macroscopic time points, sample the solution
$$\tilde{x}^{n+1}=\hat{x}(t)\quad\textrm{if } t=(n+1)\Delta T$$
\item Repeat from 1 for the next macro time step integration.
\end{enumerate}

For our method, the overall reduction in the computation is determined by $\alpha_1$ and $\alpha_2$. If we use a uniform time step $\delta t$ with an explicit Euler for (\ref{eq:VSSHMM:model2}), the computational savings of our method using the explicit Euler method is of order $\mathcal{O}(\alpha_1+\alpha_2)$.

\section{Numerical Examples}\label{sec:numerical}
In this section, we test VSHMM for several test problems. The first two test problems, which are dissipative and high oscillatory, are designed to be simple to analyze but have non-trivial behaviors through nonlinear interactions between different scales. In the last part, VSHMM is applied for the time integration of multiscale partial differential equations whose solutions can be represented by a small number of basis functions regardless of the finest scale of the problems. 

\subsection{Dissipative case}\label{subsec:nume:dissipative}
As the first example, we consider the following ODE system
\begin{equation}\label{eq:exp1}
\begin{split}
\frac{d\xi}{dt}&=\sin(\xi+\eta+\zeta)-\frac{(\xi+\eta+\zeta)^2}{20},\\
\frac{d\eta}{dt}&=\frac{1}{\epsilon}(3\xi^2-\eta^2+\zeta^2),\\
\frac{d\zeta}{dt}&=\frac{1}{\epsilon^2}(\xi-\eta-\zeta),\\
(\xi_0,\eta_0,\zeta_0)&=(5,-10,5).
\end{split}
\end{equation}
$\xi$ is the slow variable and $\eta$ and $\zeta$ are fast variables which are dissipative and converge to steady states for a fixed $\xi$. The smallest scale of order $\frac{1}{\epsilon^2}$, $\zeta$, quickly converges to $\xi-\eta$ for fixed $\xi$ and $\eta$. Using this stationary value of $\zeta$, we can check that $\eta$ also converges to $2\xi$ for fixed $\xi$. Thus, as $\epsilon\to 0$, the averaged equation which approximates the dynamics of $\xi$ is
\begin{equation}\label{eq:avgeq}
\frac{d\Xi}{dt}=\sin(2\Xi)-\frac{\Xi^2}{5}
\end{equation}
where the initial value (due to the ergodicity of the fast scales $\eta$ and $\zeta$) is given by $\Xi(0)=\xi_0=5$.

We check how well the averaged solution $\Xi$ approximates the slow variable. For $\epsilon=10^{-2}$, figure \ref{fig:exp1a} shows the averaged solution computed by a 4th order Runge-Kutta method with a time step $\Delta t=10^{-2}$ along with the direct numerical solution using the same 4th order Runge-Kutta method but with a much smaller time step $\Delta t=10^{-6}$. The averaged solution is on top of the DNS solution which implies that the averaged solution is a good approximation to the slow variable. Figure \ref{fig:exp1b} shows the solutions using splitting of vector fields with variable time steps (VSHMM, green) and constant time steps (red). The variable time step at the beginning is fine enough to resolve the fast transient behavior of $\eta$ and $\zeta$ to the invariant measure. VSHMM is comparable to the DNS solution (or the averaged solution) while the constant time step deviates from the DNS solution significantly at the beginning of the simulation. For our method, the averaged time steps are $10^{-2}$ and $10^{-3}$ for $h_1(t)$ and $h_2(t)$ respectively while $\delta t$ is fixed at $10^{-4}$, that is $\alpha_1=\frac{10^{-2}}{10^{-4}}=100$ and $\alpha_2=\frac{10^{-3}}{10^{-4}}=10$. Thus, for a $\mathcal{O}(1)$ time, the computational cost of our method is comparable to the simulation of the averaged equation.
\begin{figure}[h]
	\centering
	\subfloat[DNS of $\xi$ (blue) and averaged (green) solutions\label{fig:exp1a}]{\includegraphics[width=.45\textwidth]{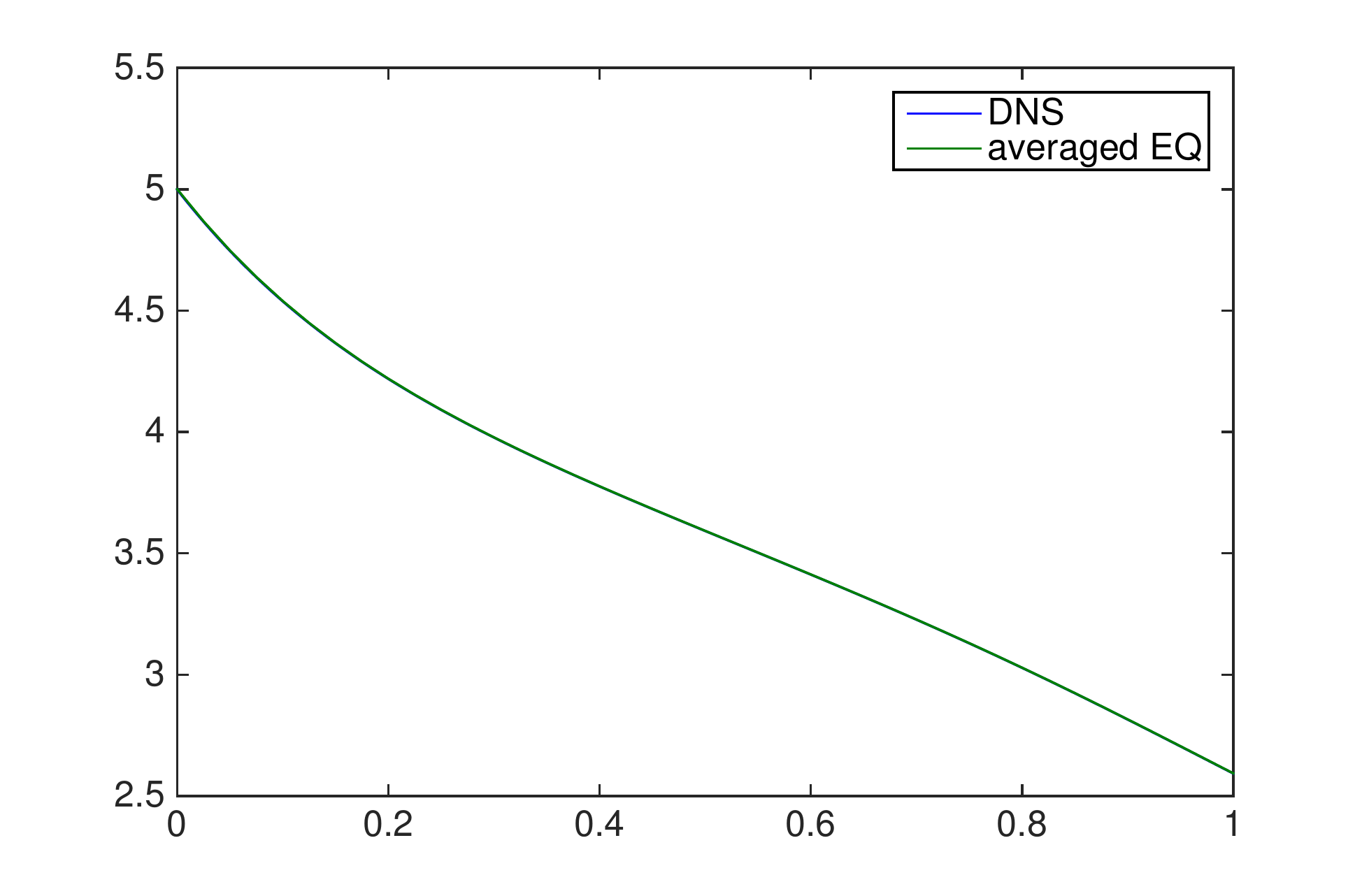}}
	\subfloat[DNS of $\xi$ (blue) and splitting method with variable (green) and constant (red) time steps\label{fig:exp1b}]{\includegraphics[width=.45\textwidth]{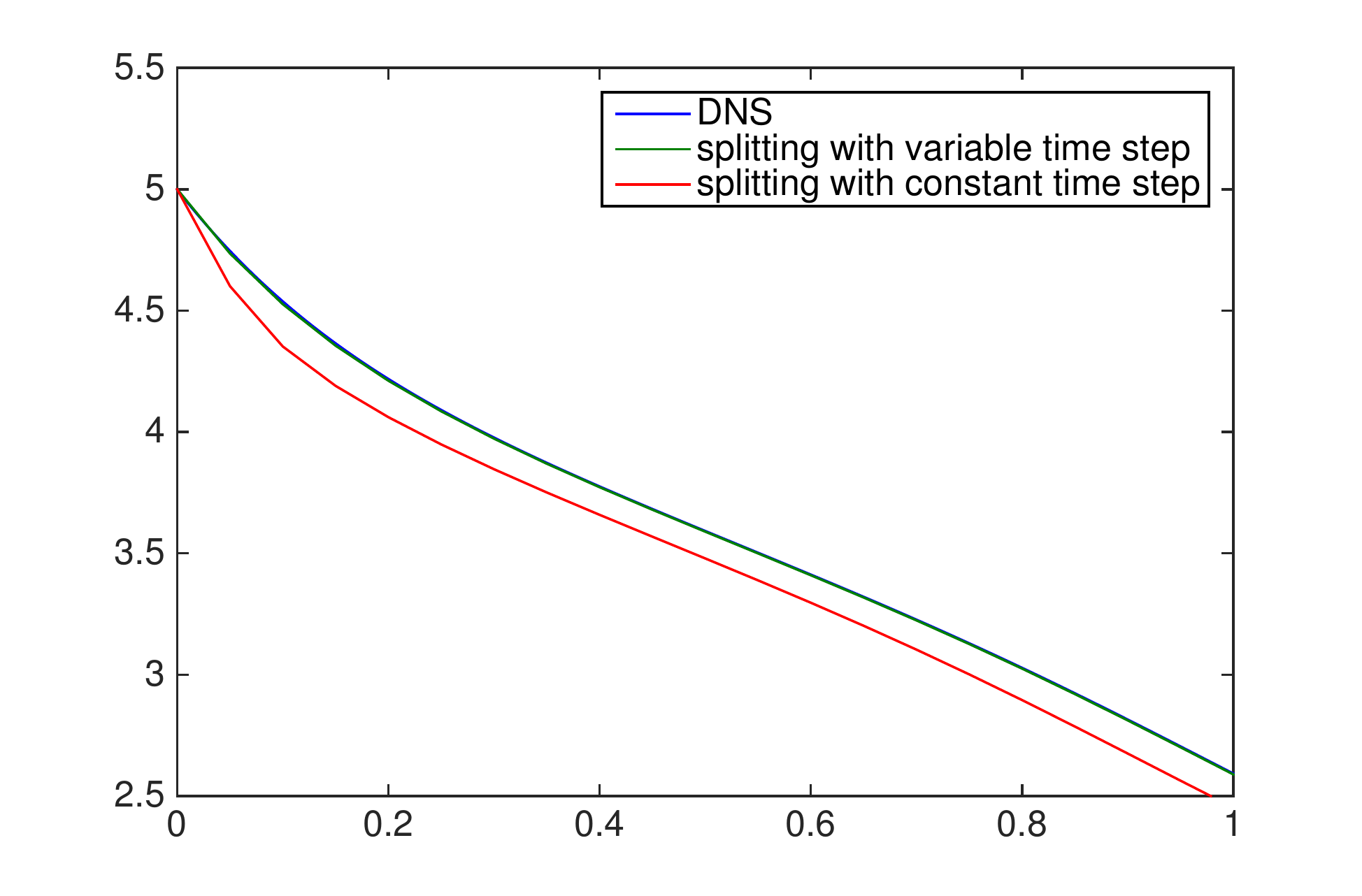}}
	\caption{Numerical solutions of \eqref{eq:exp1} using various methods, direct numerical simulation (DNS), averaged equation, and splitting technique with variable or constant time steps.}
	\label{fig:exp1_avg}
\end{figure}

\subsection{Oscillatory case}\label{subsec:nume:oscillatory}
For the second test problem, we consider the following highly oscillatory problem which is discussed in \cite{3scHMM} as a test problem with several temporal scales
\begin{equation}
\left\{\begin{split}
	\frac{dx_1}{dt}=&-\frac{1}{\epsilon^2}y_1+\frac{1}{\epsilon}y_2^2-3x_1x_2^2,\\
	\frac{dx_2}{dt}=&-\left(\frac{1}{\epsilon^2}+\frac{1}{\epsilon}\right)y_2-x_2,\\
	\frac{dy_1}{dt}=&\frac{1}{\epsilon^2}x_1+\frac{1}{2}y_1,\\
	\frac{dy_2}{dt}=&\left(\frac{1}{\epsilon^2}+\frac{1}{\epsilon}\right)x_2-y_2+2x_1^2y_2.
\end{split}
\right.
\end{equation}
The equation describes a system of two coupled harmonic oscillators with resonance and all four variables have the $\frac{1}{\epsilon^2}$ time scale and an iterated application of the two-scale time integrators are successfully applied for this problem. To check the slow variables, the following new variables were introduced \cite{3scHMM}
\begin{equation}
\begin{split}
I_1=&x_1^2+y_1^2,\\
I_2=&x_2^2+y_2^2,\\
\theta=&x_1x_2+y_1y_2,\\
\cos\phi_1=&\frac{x_1}{\sqrt{I_1}}.
\end{split}
\end{equation}
The first two variables, $I_i,i=1,2,$ correspond to the squared radial distance of each oscillators in the polar coordinate; $\theta$ is a polynomial variable which describes the resonance between two oscillators; $\phi_1$ is the angle of the first oscillator. A direct calculation shows that the time derivatives of the new variables are given by
\begin{equation}
\begin{split}
\frac{dI_1}{dt}=&\frac{2}{\epsilon}x_1y_2^2-6x_1^2x_2^2+y_1^2,\\
\frac{dI_2}{dt}=&-2I_2+4x_1^2y_2^2,\\
\frac{d\theta}{dt}=&\frac{1}{\epsilon}(x_2y_2^2+y_1x_2-x_1y_2)+(-y_1y_2/2-x_1x_2-3x_1x_2^3+2x_1^2y_1y_2),\\
\frac{d\phi_1}{dt}=&\frac{1}{\epsilon^2}.
\end{split}
\end{equation}
Interesting fact is that the average value of $x_1y_2^2$ for a time interval larger than $\mathcal{O}(\epsilon)$ is of order $\epsilon^2$ \cite{3scHMM}. Therefore, the averaged time derivative of $I_1$ is bounded independent of $\epsilon$.

Figure \ref{fig:VSSHMM:osc} shows the results from DNS and splitting method with constant time steps (figure \ref{fig:VSSHMM:osc} (a)) and with VSHMM (figure \ref{fig:VSSHMM:osc} (b)) for $\epsilon=10^{-3}$. For DNS we use a 4th order Runge-Kutta with a time step $\delta t=10^{-7}$.
For variable time stepping, the cosine kernel $K(t)=\left(1+\cos({2\pi (t-1/2)})\right)$ is used with the same 4th order Runge-Kutta method for the integration of each scale vector field. The savings factors are $\alpha_1=68,000$ and $\alpha_2=330$. Note that the VSHMM solutions are plotted at a sampling interval $\Delta T=8.00\times 10^{-1}$ larger than the average value of the largest time step $h_1$, $6.8\times 10^{-3}$.
\begin{figure}
\centering
\subfloat[constant time steps\label{fig:exp2const}]{\includegraphics[width=0.45\textwidth]{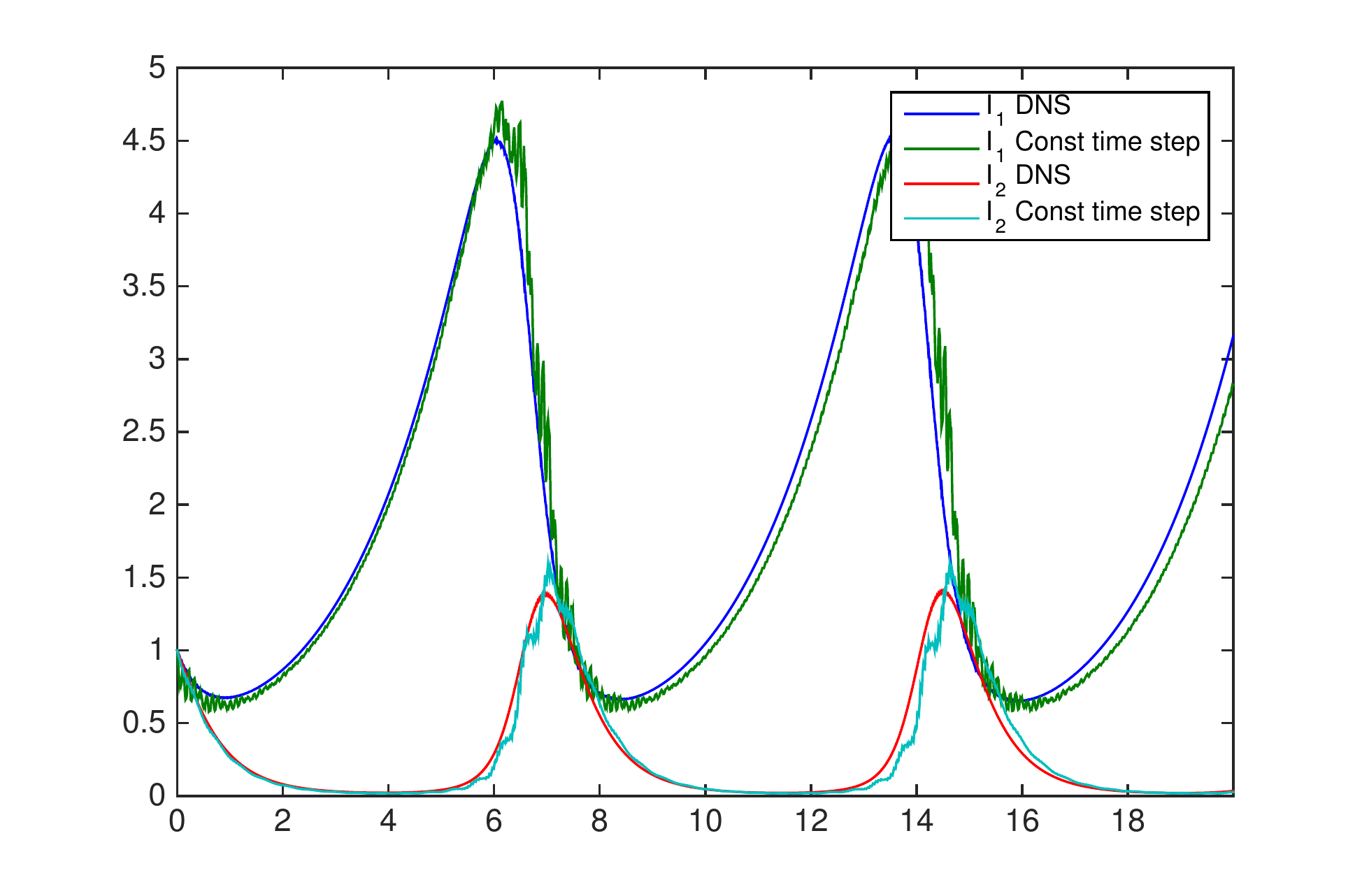}}
\subfloat[variable time steps\label{fig:exp2variable}]{\includegraphics[width=0.45\textwidth]{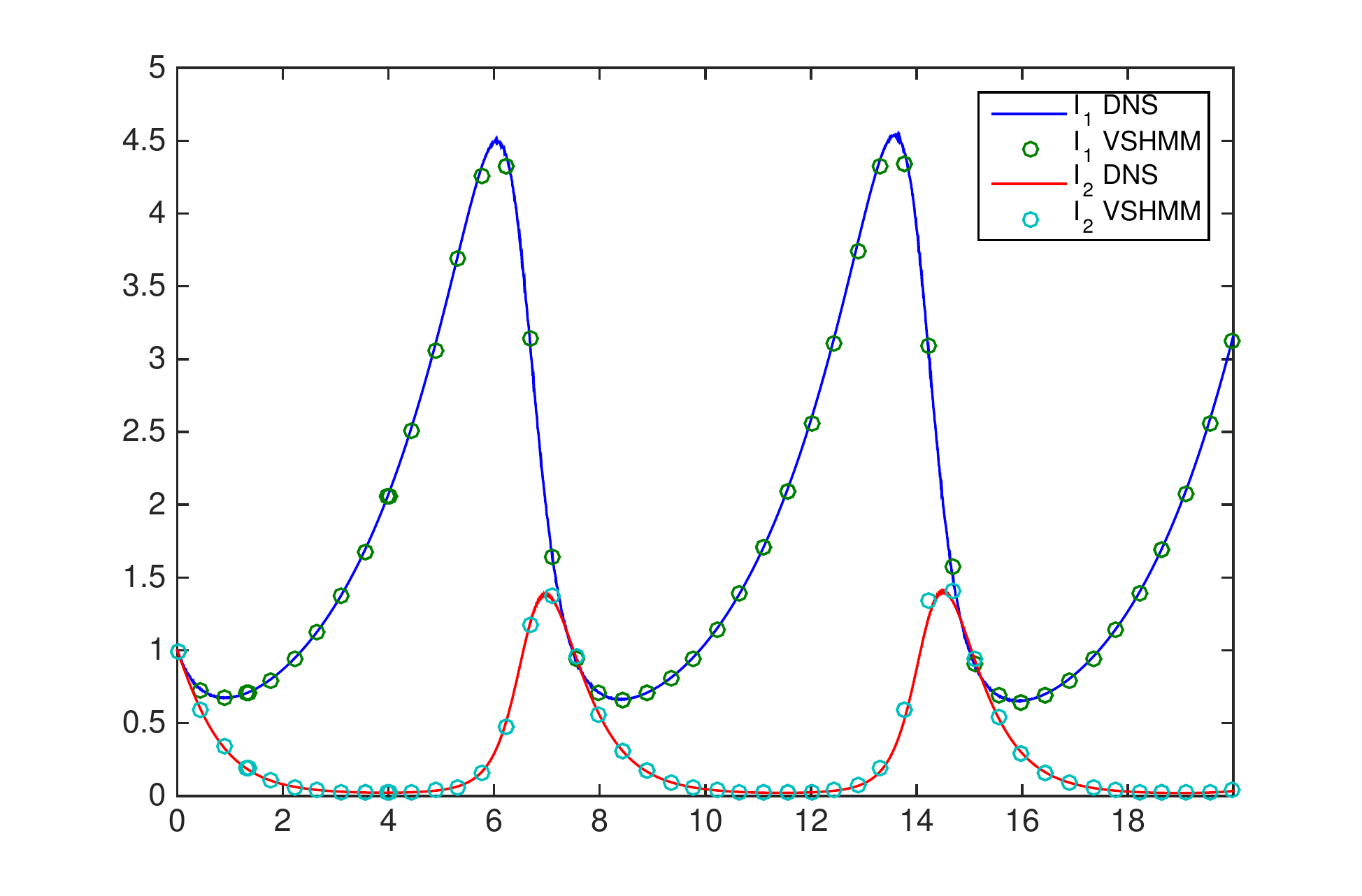}}
\caption{Numerical solutions of the slow variables $I_1$ and $I_2$ using (a) constant and (b) variable step size techniques. Simulation parameters are $\epsilon=10^{-3}, \alpha_1=6.82\times 10^{4}, \alpha_2=3.30\times 10^{2}, \delta t=4.40\times 10^{-7} \textrm{ and } \Delta T=8.00\times 10^{-1}$}\label{fig:VSSHMM:osc}
\end{figure}

The solution using constant time steps has errors in the phase and the amplitude from the DNS solution and the phase errors become larger for longer times. On the other hand, the solution using variable time steps is on top of the DNS solution without a phase error even for a long time.

\subsection{Application to sparse spectral methods}
In this last part, we apply VSHMM for the integration of multiscale PDEs with sparsity in the solution space, that is, the sparse spectral methods \cite{sparseFFT,sparsePDE}. For the solution $u$ to the following general time dependent problem
\begin{equation}\label{eq:general}
u_t=F(u),
\end{equation}
standard spectral methods approximate the solution using a given basis $\{\psi_k\}$ and time dependent coefficients $\{\alpha_k\}$
\begin{equation}
u=\sum_k\alpha_k(t)\psi_k(x).
\end{equation}
For a smooth solution $u$, the spectral approximation has a spectral accuracy high than other finite difference or finite element approximations \cite{SpectralMethod}. With periodic boundary conditions, the natural basis is the Fourier basis $\psi_k(x)=\exp(ikx)$ and the time integration of the equation \eqref{eq:general} is achieved by the integration of each coefficient $\alpha_k(t)=\hat{u}_k(t)$, i.e., the Fourier transform, 
\begin{equation}
\frac{d u_k}{dt}=\hat{F}_k
\end{equation}
where $\hat{F}_k$ is the Fourier transform of $F(u)=F(\sum_k \hat{u}_k\exp(ikx))$. The sparse spectral method \cite{sparseFFT,sparsePDE} uses only a small number of basis to reconstruct the solution $u$. This implies that only a small number of Fourier modes, $u_k$, need to be updated in time. The key observation in applying VSHMM for the sparse spectral method is the clustering of the Fourier modes. The clustering of the Fourier modes gives a natural way to decompose the force $F(u)$ into different scale components. 

%
%
%
%
\subsubsection*{Multiscale diffusion equation}
The first example is a diffusion equation with no external force
\begin{equation}\label{eq:diffusion}
\partial_t u=\partial_x\left(d(x)\partial_x u\right)
\end{equation}
where $d$ is a space dependent diffusion coefficient. When $d(x)$ contains multiple scales, this problem becomes computationally expensive as the finest scale of $d(x)$ requires a large number of grids for accuracy and a tiny time step for stability. For the following multiscale diffusion coefficient
\begin{equation}
d(x)=\frac{1}{4}\left(\exp(\sin(64x))+\exp(\sin(256x))\right),
\end{equation}
a numerical test with the Fourier spectral method and a third order Runge-Kutta time integration shows that at least 2048 Fourier modes are required and the time step must be shorter than $1\times 10^{-6}$ to have a consistent result with the converged solution.

Using a smooth initial profile centered at $x=\pi$
\begin{equation}\label{eq:initial}
u(0,x)=\exp(-(x-\pi)^2),
\end{equation}
the numerical solution with 2048 Fourier modes and a time step $1\times 10^{-6}$ and the time series of the log-scale magnitude of the Fourier coefficients, $\log|\hat{u}_k|$, are shown in figure \ref{fig:dif:DNS}. As there is no external forcing and with the periodic boundary condition, the solution converges to a constant value $0.2821=\frac{1}{2\pi}\int_0^{2\pi}u(0,x)dx$, the mean of the initial value. For the diffusion problem, high wavenumbers decay fast but some high wavenumbers which interact with the multiscale diffusion coefficient $d$ decay slow. From the log-scale magnitude of the Fourier coefficients of the solution (figure \ref{fig:dif:DNS:log}), we can check that only a small number of Fourier modes have significantly large magnitude. This implies that the solution can be reconstructed efficiently using a small subset of the basis. 
\begin{figure}[h!]
	\centering
	\subfloat[Time series of $u$ in physical space]{\includegraphics[width=.49\textwidth]{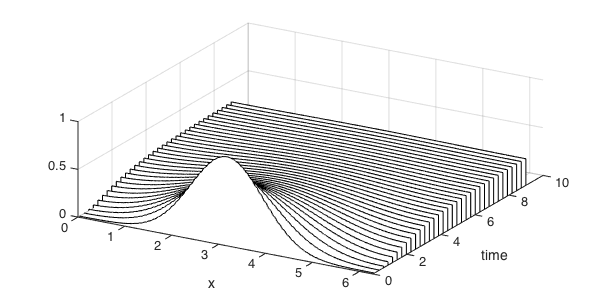}}
	\subfloat[Time series of $\log|\hat{u}_k|$\label{fig:dif:DNS:log}]{\includegraphics[width=.49\textwidth]{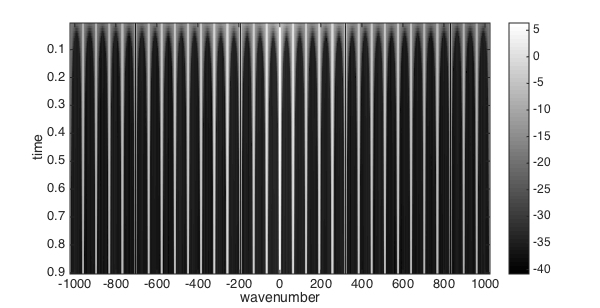}}	
	\caption{DNS solution of the multiscale diffusion equation \eqref{eq:diffusion} with 2048 Fourier modes and a time step $\Delta t=1\times 10^{-6}$ }
	\label{fig:dif:DNS}
\end{figure}

The sparsity of the solution in some basis is the fundamental idea of the sparse spectral methods \cite{sparseFFT,sparsePDE}. The method employed in \cite{sparsePDE} to compress the solution is a soft thresholding with a shrinkage value $\lambda$. That is, for a function $v$, the sparse representation of $v$ is obtained through the following soft thresholding
\begin{equation}
\hat{v}^s_k=\max(|\hat{v}_k|-\lambda,0)\frac{\hat{v}_k}{|\hat{v}_k|}.
\end{equation}

\begin{figure}[h!]
	\centering
	\includegraphics[width=.99\textwidth]{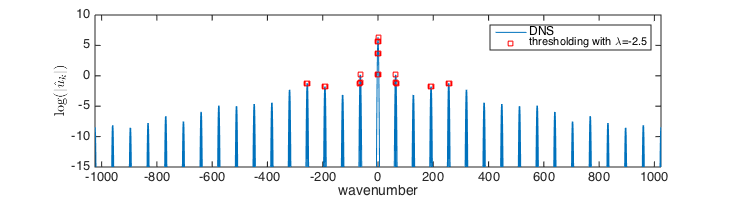}
	\caption{Log-scale magnitude of the Fourier coefficients, $\log(|\hat{u}_k|)$ by DNS (blue) and after soft thresholding with $\lambda=-2.5$ (red)}
	\label{fig:dif:coef}
\end{figure}
For the solution of \eqref{eq:diffusion}, a soft thresholding with $\lambda=10^{-2.5}$ retain only a very small number of coefficients (marked with red squares in figure \ref{fig:dif:coef}). By collecting adjacent wavenumbers, we obtain the following four clusters
\begin{eqnarray}
\nonumber K_0&=&\{k|k=0, \pm 1, \pm 2, \pm 3\},\\
\label{eq:sparse:dif:cluster} K_1&=&\{k|k=\pm 62, \pm 63, \pm 65, \pm 66\},\\
\nonumber K_2&=&\{k|k=\pm 192, \pm 193\},\\
\nonumber K_3&=&\{k|k=\pm 255, \pm 257\}.
\end{eqnarray}
The sparse spectral method achieves efficiency using a sparse representation of the solution yet it still suffers from a small time step restricted by the wavenumbers in $K_3$ for the stability in explicit time integrations. The key observation in applying VSHMM for the sparse spectral method to speed up the time integration is the clustering of Fourier modes \eqref{eq:sparse:dif:cluster}. Once we evaluate $F(u)=\partial_x(d(x)\partial_x u)$ at each time step, the clustering determines the decomposition of $F(u)$. With the four clusters \eqref{eq:sparse:dif:cluster}, we decompose $F(u)$ into four parts $F_j(u),j=0,1,...,3$ where $F_j(u)$ is defined as
$$(\hat{F_j})_k=\left\{\begin{array}{cc}\hat{F}_k&\mbox{if }k\in K_j,\\0&\mbox{otherwise}.\end{array}\right.$$
where it is easily verified that $F(u)=\sum_j F_j(u)$. Now we apply VSHMM by treating $F_j$ as different scale components. 

The VSHMM solution using a time step $\delta t = 1\times 10^{-5}$ for the finest scale $F_3$ with $\alpha_1=150,\alpha_2=18, \alpha_3=1.5$ is shown in figure \ref{fig:dif:DNSvsVSHMM} along with the DNS solution. 
\begin{figure}[h!]
	\centering
	\includegraphics[width=1\textwidth]{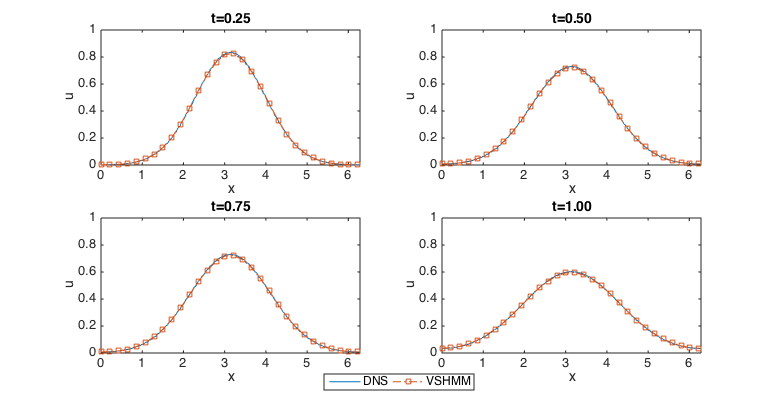}
	\caption{DNS and VSHMM solutions of the multiscale diffusion equation \eqref{eq:diffusion} at $t=0.25, 0.50, 0.75$ and $1.00$}
	\label{fig:dif:DNSvsVSHMM}
\end{figure}
The VSHMM solution is on top of the DNS solution while VSHMM is more than 60 times faster than the standard explicit integration of the sparse spectral method and 1000 times faster than the direct numerical simulations.

%
%
%
%
\subsubsection*{Multiscale advection equation}
Another example for the fast time integration of sparse spectral method is an advection equation which shows less discrete coefficients. For $u:(0,\infty)\times [0,2\pi)\to \mathbb{R}$, the following advection equation with no external forcing 
\begin{equation}\label{eq:advection}
\partial_t u+a(x)u_x=0, \qquad u(0,x)=u_0
\end{equation}
describes the transport of the initial value $u_0$ under a multiscale velocity field $a(x)$\begin{equation}
a=\frac{1}{4}\exp\left(\frac{0.6+0.2\cos(3x)}{1+0.35\sin(64x)+0.35\sin(256x)}\right).
\end{equation}

A direct numerical simulation result up to $t=36$ using 1024 Fourier modes and a third order Runge-Kutta time integration with a time step $\Delta t=1.6\times 10^{-4}$ is shown in figure \ref{fig:adv:DNS}. As time flows, the solution moves to the left while its shapes are slightly changing due to inhomogeneous velocity field. From the time series of the magnitude of the Fourier coefficients (figure \ref{fig:adv:DNS} (b)), we can check that this problem has more continuous range of Fourier coefficients than the diffusion problem. 
\begin{figure}[h!]
	\centering
	\subfloat[Time series of $u$ in physical space]{\includegraphics[width=.49\textwidth]{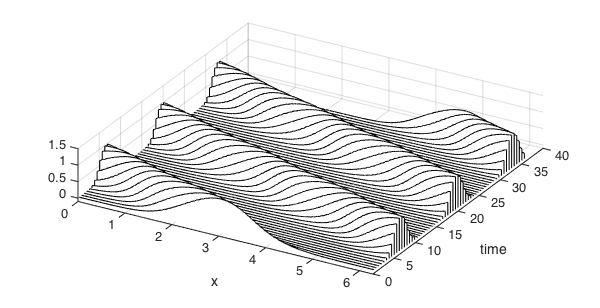}}
	\subfloat[Time series of $\log|\hat{u}_k|$]{\includegraphics[width=.49\textwidth]{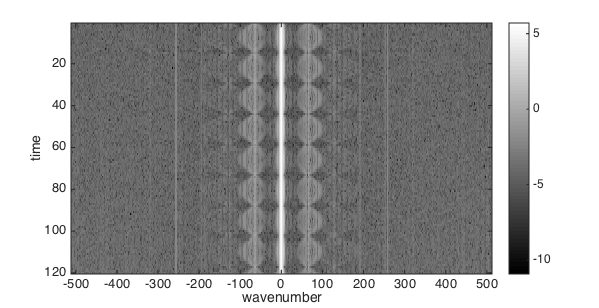}}	
	\caption{DNS solution of the multiscale advection equation \eqref{eq:advection} with 1024 Fourier modes and a time step $\Delta t=1.6\times 10^{-4}$ }
	\label{fig:adv:DNS}
\end{figure}
For the solution $u$ of \eqref{eq:advection}, a soft thresholding with $\lambda=10^{-2.5}$ tells that the following four Fourier mode clusters are crucial to represent the solution $u$
\begin{eqnarray}
\nonumber K_0&=&\{k|k=0\pm 24\},\\
\label{eq:sparse:adv:cluster} K_1&=&\{k|k=\pm 64 \pm 24\},\\
\nonumber K_2&=&\{k|k=\pm 138 \pm 24\},\\
\nonumber K_3&=&\{k|k=\pm 256 \pm 10\}.
\end{eqnarray}
where some extra wavenumbers are added to each cluster as buffers.

Figure \ref{fig:adv:DNSvsVSHMM} shows the solution by DNS (blue) and VSHMM (red) at $t=9,18,27$ and 36. The VSHMM solution captures not only the correct propagation speed but also the shape of the solution. For this result, the VSHMM solution uses a time step $\delta = 1\times 10^{-3}$ for the finest scale with $\alpha_1=26,\alpha_2=9,\alpha_3=2$. Using these time step values, VSHMM uses 40 times less iterations than DNS.
\begin{figure}[h!]
	\centering
	\includegraphics[width=1\textwidth]{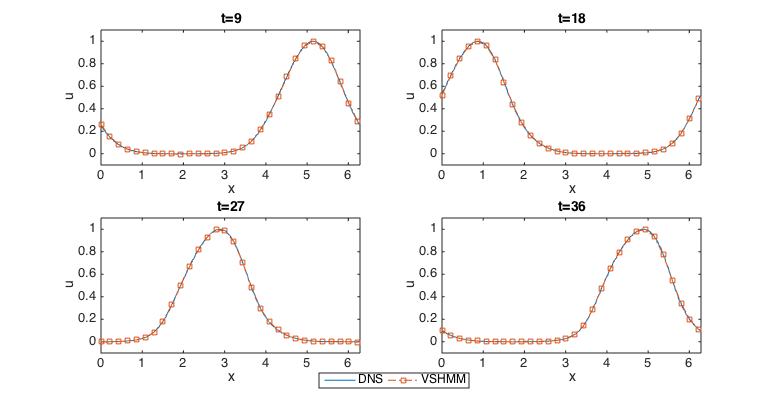}
	\caption{DNS and VSHMM solutions of the multiscale advection equation \eqref{eq:advection} at $t=9,18,27$ and 36.}
	\label{fig:adv:DNSvsVSHMM}
\end{figure}

\section{Conclusion}\label{sec:conclusion}
The variable step size Heterogeneous Multiscale Method (VSHMM) is extended to several temporal scale problems without using iterated applications of two-scale integrators. By splitting the vector field of different scales, VSHMM uses time steps corresponding to different scales; it further uses variable time steps (i.e., time dependent time steps) for 1) high accuracy and 2) fast (Monte-Carlo like) integration  of invariant measure. The computational complexity of our method increases linearly proportional the number of different scale which is much faster than the iterated application of two-scale methods which have exponentially increasing computational complexity. The method is tested for dissipative and oscillatory problems and applied to sparse spectral methods for multiscale PDEs. 

For dissipative systems, under appropriate assumptions, we show that the several scale problems can be treated as two-scale problems (Theorem \ref{thm:dissipative}) which justifies the use of different time steps for different scale components of the vector field. For high oscillatory problems, whose invariant measure is not of a Dirac type, we had to relax the interactions between the fast variables; no direct interactions between the fast variables except indirect interactions through the slow dynamics. Under this relaxation, it was straightforward to treat the different fast variables as one scale fast variables which is equivalent to Theorem \ref{thm:dissipative} of the dissipative case. But there is still another problem related to the integration of the effective force and thus the variable step sizes are employed for intermediate mesoscopic scales. 

Our proposed variable time stepping can be extended for time dependent partial differential equations including nonlinear problems with continuous temporal scales. This extension results in an explicit time integration method using a large time step violating the CFL condition and it will be reported in an upcoming paper \cite{FastExplicit}. 

\section*{Acknowledgments}
The authors thank Seongjun Kim for comments and a careful reading of the manuscript. The research was partially supported by NSF grant DMS-1217203.

\end{document}